\documentclass[11pt]{amsart}
\usepackage{amsmath}
\usepackage[active]{srcltx}
\usepackage{t1enc}
\usepackage[latin2]{inputenc}
\usepackage{verbatim}
\usepackage{amsmath,amsfonts,amssymb,amsthm}
\usepackage[mathcal]{eucal}
\usepackage{enumerate}
\usepackage[centertags]{amsmath}
\usepackage{graphics}

\newtheorem{theorem}{Theorem}

\newtheorem*{weisz}{Theorem A}
\newtheorem*{weisz2}{Theorem B}
\newtheorem*{weisz3}{Theorem C}

\newtheorem*{est}{Theorem MTK}

\begin{document}
\title[]{ Conjugate Transforms on Dyadic Group}
\author{ Ushangi Goginava and Aleksandre Saatashvili}
\address{I. Vekua Institute of Applied Mathematics and Faculty of Exact and
Natural Sciences of I. Javakhishvili Tbilisi State University, Tbilisi 0186,
2 University str., Georgia}
\email{zazagoginava@gmail.com}
\address{Aleksandre Saatashvili, Massachusetts Institute of Technology, 77
Massachusetts Ave, Cambridge, MA 02139, United States}
\email{saata@mit.edu}
\thanks{.}

\begin{abstract}
In this paper we study the properties of the Lebesgue constant of the
conjugate transforms. For conjugate Fejér means we will find necessary and
sufficient condition on $t$ for which the estimation $E\left\vert \widetilde{%
\sigma }_{n}^{\left( t\right) }f\right\vert \lesssim E\left\vert
f\right\vert $ holds . We also prove that for dyadic irrational $t$, $L\log
L $ is maximal Orlicz space for which the estimation $E\left\vert \widetilde{%
\sigma }_{n}^{\left( t\right) }f\right\vert \lesssim 1+E\left( \left\vert
f\right\vert \log ^{+}\left\vert f\right\vert \right) $ is valid.
\end{abstract}

\maketitle

\bigskip \footnotetext{%
2010 Mathematics Subject Classification. 42C10, 60G42.
\par
Key words and phrases: Conjugate Walsh transform, Martingale transform,
convergence in norm.}

\section{Dyadic Hardy spaces and conjugate transforms}

Let $\mathbb{P}$ denote the set of positive integers, $\mathbb{N}\mathbf{:=}%
\mathbb{P}\mathbf{\cup \{}0\mathbf{\},}$ the set of all integers by$\,\,%
\mathbb{Z}$ and the set of dyadic rational numbers in the unit interval $%
\mathbb{I}:=[0,1)$ by $\mathbb{Q}$. In particular, each element of $\mathbb{Q%
}$ has the form $\frac{p}{2^{n}}$ for some $p,n\in \mathbb{N},\,\,\,0\leq
p\leq 2^{n}$.

Denote by $Z_{2}$ the discrete cyclic group of order 2, that is $%
Z_{2}=\{0,1\},$ where the group operation is the modulo 2 addition and every
subset is open. The Haar measure on $Z_{2}$ is given such that the measure
of a singleton is 1/2. Let $G$ be the complete direct product of the
countable infinite copies of the compact groups $Z_{2}.$ The elements of $G$
are of the form $x=\left( x_{0},x_{1},...,x_{k},...\right) $ with $x_{k}\in
\{0,1\}\left( k\in \mathbb{N}\right) .$ The group operation on $G$ is the
coordinate-wise addition, the measure (denote\thinspace $\,$by$\,\,\mu $)
and the topology are the product measure and topology. The compact Abelian
group $G$ is called the Walsh group. A base for the neighborhoods of $G$ can
be given in the following way: 
\begin{eqnarray*}
I_{0}\left( x\right) &:&=G,\,\,\,I_{n}\left( x\right) :=\,I_{n}\left(
x_{0},...,x_{n-1}\right) \\
&:&=\left\{ y\in G:\,y=\left( x_{0},...,x_{n-1},y_{n},y_{n+1},...\right)
\right\} ,
\end{eqnarray*}%
\begin{equation*}
\,\left( x\in G,n\in \mathbb{N}\right) .
\end{equation*}%
These sets are called the dyadic intervals. Let $0=\left( 0:i\in \mathbb{N}%
\right) \in G$ denote the null element of $G,\,\,\,I_{n}:=I_{n}\left(
0\right) \,\left( n\in \mathbb{N}\right) .$ For every finite set $E$ the
number of elents in $E$ we denote by $\left\vert E\right\vert $, i. e. $%
\left\vert E\right\vert :=\left( E\right) ^{\#}.$

For $k\in \mathbb{N}$ and $x\in G$ denote 
\begin{equation*}
r_{k}\left( x\right) :=\left( -1\right) ^{x_{k}}\,\,\,\,\,\,
\end{equation*}%
the $k$-th Rademacher function on dyadic group $G$.

Denote the dyadic expension of $t\in \mathbb{I}$ by%
\begin{equation*}
t=\sum\limits_{j=0}^{\infty }\frac{t_{j}}{2^{j+1}},t_{j}=0,1.
\end{equation*}%
In the case of $t\in \mathbb{Q}$ choose the expension which terminates in
zeros. For $t\in \mathbb{I}$ we denote $\rho \left( t\right) :=\left(
t_{0},t_{1},...\right) \in G.$

The $\sigma $-algebra generated by the dyadic intervals $\left\{ I_{n}\left(
x\right) :x\in G\right\} $ is denoted by $A^{n},$ more precisely,%
\begin{equation*}
A^{n}:=\sigma \left\{ I_{n}\left( x\right) :x\in G\right\} .
\end{equation*}

The expectation and the conditional expectation operators relative to $A^{n}$
$\left( n\in \mathbb{N}\right) $ are denoted by $E$ and $E_{n}$,
respectively.

The norm (or quasinorm) of the space $L_{p}$ is defined by 
\begin{equation*}
\left\Vert f\right\Vert _{p}:=\left( E\left\vert f\right\vert ^{p}\right)
^{1/p}\,\,\,\,\left( 0<p<+\infty \right) .
\end{equation*}

Denote by $f=\left( f^{\left( n\right) },n\in \mathbb{N}\right) $ martingale
with respect to $\left( A^{n},n\in \mathbb{N}\right) $ (for details see, e.
g. \cite{Webook1, Webook2}). The maximal function of a martingale $f$ is
defined by 
\begin{equation*}
f^{\ast }=\sup\limits_{n\in \mathbb{N}}\left\vert f^{\left( n\right)
}\right\vert .
\end{equation*}

In case $f\in L_{1}$, the maximal function can also be given by 
\begin{equation*}
f^{\ast }=\sup\limits_{n\in \mathbb{N}}\left\vert E_{n}f\right\vert .
\end{equation*}

For $0<p<\infty $ the Hardy martingale space $H_{p}$ consists of all
martingales for which

\begin{equation*}
\left\| f\right\| _{H_{p}}:=\left\| f^{*}\right\| _{p}<\infty.
\end{equation*}

For a martingale 
\begin{equation*}
f\sim \sum\limits_{n=0}^{\infty }\left( f^{\left( n\right) }-f^{\left(
n-1\right) }\right)
\end{equation*}%
the conjugate transforms are defined by 
\begin{equation*}
\widetilde{f}^{\left( t\right) }\sim \sum\limits_{n=0}^{\infty }r_{n}\left(
\rho \left( t\right) \right) \left( f^{\left( n\right) }-f^{\left(
n-1\right) }\right) ,
\end{equation*}%
where $t\in \mathbb{I}$ is fixed.

Note that $\widetilde{f}^{\left( 0\right) }=f.$ As is well known, if $f$ is
an integrable function, then conjugate transforms $\widetilde{f}^{\left(
t\right) }$ do exist almost everywhere, but they are not integrable in
general.

The following equation holds (\cite{Webook1, Webook2}) 
\begin{equation*}
\left\Vert \widetilde{f}^{\left( t\right) }\right\Vert _{H_{p}}=\left\Vert
f\right\Vert _{H_{p}}\text{ \ \ }\left( 0<p<\infty ,t\in \mathbb{I}\right) .
\end{equation*}%
Furthermore, Khintchines inequality implies that%
\begin{equation*}
\left\Vert f\right\Vert _{H_{p}}^{p}\sim \int\limits_{\mathbb{I}}\left\Vert 
\widetilde{f}^{\left( t\right) }\right\Vert _{H_{p}}^{p}dt\text{ \ \ \ \ }%
\left( 0<p<\infty \right) .
\end{equation*}

Let $Q\left( L\right) =Q\left( L\right) (\mathbb{I})$ be the Orlicz space 
\cite{KR} generated by Young function $Q$, i.e. $Q$ is convex continuous
even function such that $Q(0)=0$ and

\begin{equation*}
\lim\limits_{u\rightarrow +\infty }\frac{Q\left( u\right) }{u}=+\infty
,\,\,\,\,\lim\limits_{u\rightarrow 0}\frac{Q\left( u\right) }{u}=0.
\end{equation*}

This space is endowed with the norm 
\begin{equation*}
\Vert f\Vert _{Q\left( L\right) (\mathbb{I})}=\inf \{k>0:\int\limits_{%
\mathbb{I}}Q(\left\vert f\right\vert /k)\leq 1\}.
\end{equation*}

In particular, if $Q(u)=u\log (1+u),u>0$ then the corresponding space will
be denoted by $L\log ^{+}L(\mathbb{I})$.

\section{Walsh system and conjugate Fejér means}

Let $m\in \mathbb{N}$, then $m=\sum\limits_{i=0}^{\infty }m_{i}2^{i},$ where 
$m_{i}\in \{0,1\}\,\,\left( i\in \mathbb{N}\right) $, i.e. $m$ is expressed
in the number system of base 2. Denote $\left\vert m\right\vert :=\max
\{j\in \mathbb{N}\mathbf{:}m_{j}\neq 0\}$, that is, $2^{\left\vert
m\right\vert }\leq m<2^{\left\vert m\right\vert +1}.$

The Walsh-Paley system is defined as the sequence of Walsh-Paley functions: 
\begin{equation*}
w_{m}\left( x\right) :=\prod\limits_{k=0}^{\infty }\left( r_{k}\left(
x\right) \right) ^{m_{k}}=r_{\left\vert m\right\vert }\left( x\right) \left(
-1\right) ^{\sum\limits_{k=0}^{\left\vert m\right\vert
-1}m_{k}x_{k}}\,\,\,\,\,\,\left( x\in G,m\in \mathbb{P}\right) .
\end{equation*}

The Walsh-Dirichlet kernel is defined by 
\begin{equation*}
D_{n}\left( x\right) =\sum\limits_{k=0}^{n-1}w_{k}\left( x\right).
\end{equation*}

Recall that (\cite{S-W-S}, \cite{G-E-S}) 
\begin{equation}
D_{2^{n}}\left( x\right) =\left\{ 
\begin{array}{l}
2^{n},\mbox{if }x\in I_{n}, \\ 
0,\,\,\,\mbox{if }x\in G\backslash I_{n},%
\end{array}%
\right.  \label{dir}
\end{equation}%
and%
\begin{equation}
D_{n}\left( x\right) =w_{n}\left( x\right) \sum\limits_{k=0}^{\infty
}n_{k}\left( D_{2^{k+1}}\left( x\right) -D_{2^{k}}\left( x\right) \right) .
\label{dir2}
\end{equation}

Let $x\in I_{j}\backslash I_{j+1}$. Then from (\ref{dir2}) we have%
\begin{equation*}
D_{n}\left( x\right) =w_{n}\left( x\right) \left(
\sum\limits_{k=0}^{j-1}n_{k}2^{k}-n_{j}2^{j}\right) .
\end{equation*}

Hence,%
\begin{equation*}
\left\vert D_{n}\left( x\right) \right\vert =\alpha _{j}\left( n\right) ,
\end{equation*}%
where%
\begin{equation*}
\alpha _{j}\left( n\right) =\left\vert
\sum\limits_{k=0}^{j-1}n_{k}2^{k}-n_{j}2^{j}\right\vert .
\end{equation*}

The partial sums of the Walsh-Fourier series are defined as follows:

\begin{equation*}
S_{M}f:=\sum\limits_{i=0}^{M-1}\widehat{f}\left( i\right) w_{i},
\end{equation*}%
where the number 
\begin{equation*}
\widehat{f}\left( i\right) =E\left( fw_{i}\right)
\end{equation*}%
is said to be the $i$th Walsh-Fourier coefficient of the function\thinspace $%
f.$ It is easy to see that $E_{n}\left( f\right) =S_{2^{n}}\left( f\right) .$

For any given $n\in \mathbb{N}$ it is possible to write $n$ uniquely as%
\begin{equation*}
n=\sum\limits_{k=0}^{\infty }n_{k}2^{k},
\end{equation*}%
where $n_{k}=0$ or $1$ for $k\in \mathbb{N}$. This expression will be called
the binary expansion of $n$ and the numbers $n_{k}$ will be called the
binary coefficients of $n$.

Define the variation of an $n\in \mathbb{N}$ with binary coefficients $%
\left( n_{k}:k\in \mathbb{N}\right) $ by%
\begin{equation*}
V\left( n\right) :=\sum\limits_{k=1}^{\infty }\left\vert
n_{k}-n_{k-1}\right\vert +n_{0}.
\end{equation*}

For $n,m\in \mathbb{N}$ and $2^{N}\leq n<2^{N+1}$ we define

\begin{equation*}
T(n,m):=\left\{ i:n_{i}\neq n_{i-1},m_{i}=m_{i-1},i=0,1,...,N-1\right\} .
\end{equation*}

The Fej\'er means of Walsh-Fourier series is defined as follows%
\begin{equation*}
\sigma _{n}f:=\frac{1}{n}\sum\limits_{k=0}^{n-1}S_{k}f\text{ \ \ \ }\left(
n\in \mathbb{P}\right) .
\end{equation*}

If $f\in L_{1}$ then it is easy to show that the sequence $\left(
E_{n}\left( f\right) :n\in \mathbb{N}\right) $ is a martingale. If $f$ is a
martingale, that is $f=(f^{\left( n\right) }:n\in \mathbb{N})$ then the
Walsh-Fourier coefficients must be defined in a little bit different way: 
\begin{equation}
\widehat{f}\left( i\right) =\lim\limits_{k\rightarrow \infty }E\left(
f^{\left( k\right) }w_{i}\right) .  \label{MC}
\end{equation}

The Walsh-Fourier coefficients of $f\in L_{1}$ are the same as the ones of
the martingale $\left( E_{n}\left( f\right) :n\in \mathbb{N}\right) $
obtained from $f$.

Let%
\begin{equation*}
\beta _{0}\left( t\right) :=r_{0}\left( \rho \left( t\right) \right) ,\beta
_{k}\left( t\right) :=r_{n}\left( \rho \left( t\right) \right) \text{ if }%
2^{n-1}\leq k<2^{n}.
\end{equation*}%
Then the $n$th partial sums of the conjugate transforms is given by%
\begin{equation*}
\widetilde{S}_{n}^{\left( t\right) }f:=\sum\limits_{k=0}^{n-1}\beta
_{k}\left( t\right) \widehat{f}\left( k\right) w_{k}\text{ \ \ }\left( t\in 
\mathbb{I},n\in \mathbb{P}\right) .
\end{equation*}%
Let $2^{N}\leq n<2^{N+1}$ \ $\left( E_{-1}f=0\right) $. Then we have 
\begin{eqnarray*}
\widetilde{S}_{n}^{\left( t\right) }f &=&\rho _{0}\left( \rho \left(
t\right) \right) \widehat{f}\left( 0\right)
w_{0}+\sum\limits_{l=1}^{N}r_{l}\left( \rho \left( t\right) \right) \left(
E_{l}f-E_{l-1}f\right) \\
&&+r_{N+1}\left( \rho \left( t\right) \right) \left( S_{n}f-E_{N}f\right) \\
&=&Ef+\sum\limits_{l=1}^{N}r_{l}\left( \rho \left( t\right) \right) \left(
E_{l}f-E_{l-1}f\right) \\
&&+r_{N+1}\left( \rho \left( t\right) \right) \left( S_{n}f-E_{N}f\right) \\
&=&f\ast \widetilde{D}_{n}^{\left( t\right) },
\end{eqnarray*}%
where%
\begin{equation*}
\widetilde{D}_{n}^{\left( t\right) }=1+\sum\limits_{i=0}^{N-1}\left(
-1\right) ^{t_{i+1}}\left( D_{2^{i+1}}-D_{2^{i}}\right) +\left( -1\right)
^{t_{N+1}}\left( D_{n}-D_{2^{N}}\right) .
\end{equation*}%
It is easy to see that $\left( -1\right) ^{t_{l}}=1-2t_{l}.$ Then for $%
\widetilde{D}_{n}^{\left( t\right) }$ we can write%
\begin{eqnarray*}
\widetilde{D}_{n}^{\left( t\right) } &=&1+\sum\limits_{i=0}^{N-1}\left(
1-2t_{i+1}\right) \left( D_{2^{i+1}}-D_{2^{i}}\right) \\
&&+\left( 1-2t_{N+1}\right) \left( D_{n}-D_{2^{N}}\right) \\
&=&D_{n}-2\sum\limits_{i=0}^{N-1}t_{i+1}\left( D_{2^{i+1}}-D_{2^{i}}\right)
-2t_{N+1}\left( D_{n}-D_{2^{N}}\right) .
\end{eqnarray*}

Set%
\begin{equation*}
m:=\sum\limits_{i=0}^{N-1}t_{i+1}2^{i}<2^{N}.
\end{equation*}%
Then from (\ref{dir2}) we get%
\begin{equation*}
\widetilde{D}_{n}^{\left( t\right) }=D_{n}-2w_{m}D_{m}-2t_{N+1}\left(
D_{n}-D_{2^{N}}\right) .
\end{equation*}

The conjugate $\left( C,1\right) $-means of a martingale $f$ are introduced
by%
\begin{equation*}
\widetilde{\sigma }_{n}^{\left( t\right) }f:=\frac{1}{n}\sum%
\limits_{k=0}^{n-1}\widetilde{S}_{k}^{\left( t\right) }f\text{ \ \ \ }\left(
t\in \mathbb{I},n\in \mathbb{P}\right) .
\end{equation*}

The notiation $a\lesssim b$ in the proofs stands for $a<c\cdot b$, where $c$
is an absolute constant.

\bigskip

\section{Two Sides Estimation of Lebesgue Constant of Conjugate
Walsh-Fourier Series}

Denote by $L_{n}$ the lebesgue constants of the Walsh system:%
\begin{equation*}
L_{n}:=\int\limits_{G}\left\vert D_{n}\left( t\right) \right\vert d\mu .
\end{equation*}%
These constants were studied by many authors. For the trigonometric system
it is important to note that results of Fej\'er and Szegő, the letter gave
in \cite{Sz} an explicit formula for Lebesgue constants, namely,%
\begin{equation*}
L_{n}=\frac{16}{\pi ^{2}}\sum\limits_{k=1}^{\infty }\frac{1}{4k^{2}-1}\left(
1+\frac{1}{3}+\frac{1}{5}+\cdots +\frac{1}{2k\left( 2n+1\right) -1}\right) .
\end{equation*}

Along with the trigonometric system the Walsh-Paley system is also severely
studied for its importance in applications. The most properties of Lebesgue
constants with respect to the Walsh-Paley system were obtained by Fine in 
\cite{Fi2}. In (\cite{S-W-S}, p. 34), the two-sided estimate 
\begin{equation}
\frac{V\left( n\right) }{8}\leq L_{n}\leq V\left( n\right)  \label{v}
\end{equation}%
is proved. In \cite{luk}, Lukomskii presented the estimate $L_{n}\geq
V\left( n\right) /4$. Malykhin, Telyakovskii and Kholshchevnikova \cite{MTK}
improved the estimation (\ref{v}) and proved the following

\begin{est}
For any positive integer $n$, the two-sided inequality 
\begin{equation*}
\frac{V\left( n\right) +1}{3}\leq L_{n}<V\left( n\right)
\end{equation*}%
is valid. Here the factors $1/3$ and $1$ are sharp.
\end{est}

We would like to mention the work of Astashkin and Semenov \cite{AS} in
which the sharp two-sided estimate for Lebesgue constants with respect to
Walsh-Paley system are obtained.

In the paper \cite{tol} Toledo studied the sequence of numbers $\left\Vert
K_{n}\right\Vert _{1}$ which consists of the $L_{1}$-norm of Walsh-Fejér
kernels and he proved that 
\begin{equation*}
\sup \left\{ \left\Vert K_{n}\right\Vert _{1}:n\in \mathbb{P}\right\} =\frac{%
17}{15}.
\end{equation*}%
Denote by $L_{n}^{\left( t\right) }$ the lebesgue constants of the conjugate
transforms:%
\begin{equation*}
L_{n}^{\left( t\right) }:=\int\limits_{G}\left\vert \widetilde{D}%
_{n}^{\left( t\right) }\left( t\right) \right\vert d\mu .
\end{equation*}%
Let $n$ is positive integer and 
\begin{equation*}
m=\sum\limits_{i=0}^{N-1}t_{i+1}2^{i},2^{N}\leq n<2^{N+1}.
\end{equation*}%
Set%
\begin{equation*}
T\left( n,m\right) :=\left\{ i:n_{i}\neq n_{i-1},m_{i}=m_{i-1}\right\} .
\end{equation*}

In this section we study the properties of the Lebesgue constant of the
conjugate transforms.

\begin{theorem}
\label{TS}The two-sides inequality%
\begin{eqnarray*}
&&\max (\frac{1}{2}\left\vert T(m,n)\right\vert +\frac{1}{3}V(n)-1,\frac{1}{4%
}\left\vert T(n,m)\right\vert +\frac{2}{3}V(m)-1) \\
&\leq &L_{n}^{\left( t\right) }\leq 2V(m)+\left\vert T(n,m)\right\vert +2
\end{eqnarray*}%
is valid.
\end{theorem}

We have%
\begin{eqnarray}
L_{n}^{\left( t\right) }
&=&\sum\limits_{i=0}^{N-1}\int\limits_{I_{i}\backslash I_{i+1}}\left\vert 
\widetilde{D}_{n}^{\left( t\right) }\right\vert d\mu  \label{J1-J3} \\
&&+\int\limits_{I_{N}\backslash I_{N+1}}\left\vert \widetilde{D}_{n}^{\left(
t\right) }\right\vert d\mu +\int\limits_{I_{N+1}}\left\vert \widetilde{D}%
_{n}^{\left( t\right) }\right\vert d\mu  \notag \\
&:&=J_{1}+J_{2}+J_{3}.  \notag
\end{eqnarray}%
It is easy to see that%
\begin{equation*}
\int\limits_{I_{i}\backslash I_{i+1}}\left\vert \widetilde{D}_{n}^{\left(
t\right) }\right\vert d\mu
\end{equation*}%
\begin{equation*}
=\sum\limits_{x_{i+1}=0}^{1}\cdots
\sum\limits_{x_{N}=0}^{1}\int\limits_{I_{N+1}\left(
0,...,0,x_{i}=1,x_{i+1},...,x_{N}\right) }\left\vert \widetilde{D}%
_{n}^{\left( t\right) }\right\vert d\mu
\end{equation*}%
\begin{equation*}
=\sum\limits_{x_{i+1}=0}^{1}\cdots
\sum\limits_{x_{N}=0}^{1}\int\limits_{I_{N+1}\left(
0,...,0,x_{i}=1,x_{i+1},...,x_{N-1},x_{N}\right) }\left\vert \left(
1-2t_{N+1}\right) D_{n}-2w_{m}D_{m}\right\vert d\mu
\end{equation*}%
\begin{equation*}
=\sum\limits_{x_{i+1}=0}^{1}\cdots \sum\limits_{x_{N-1}=0}^{1}\left(
\int\limits_{I_{N+1}\left( 0,...,0,x_{i}=1,x_{i+1},...,x_{N-1},0\right)
}\right)
\end{equation*}%
\begin{equation*}
\left\vert \left( 1-2t_{N+1}\right) D_{n}-2w_{m}D_{m}\right\vert d\mu
\end{equation*}%
\begin{equation*}
+\sum\limits_{x_{i+1}=0}^{1}\cdots \sum\limits_{x_{N-1}=0}^{1}\left(
\int\limits_{I_{N+1}\left( 0,...,0,x_{i}=1,x_{i+1},...,x_{N-1},1\right)
}\right)
\end{equation*}%
\begin{equation*}
\left\vert \left( 1-2t_{N+1}\right) D_{n}-2w_{m}D_{m}\right\vert d\mu
\end{equation*}%
\begin{equation*}
=\sum\limits_{x_{i+1}=0}^{1}\cdots \sum\limits_{x_{N-1}=0}^{1}\left(
\int\limits_{I_{N+1}\left( 0,...,0,x_{i}=1,x_{i+1},...x_{N-1},0\right)
}\right)
\end{equation*}%
\begin{equation*}
\left( \left\vert \left( 1-2t_{N+1}\right) D_{n}-2w_{m}D_{m}\right\vert
+\left\vert \left( 1-2t_{N+1}\right) D_{n}+2w_{m}D_{m}\right\vert \right)
d\mu
\end{equation*}%
Since%
\begin{equation*}
\frac{\left\vert a-b\right\vert +\left\vert a+b\right\vert }{2}=\max \left\{
\left\vert a\right\vert ,\left\vert b\right\vert \right\}
\end{equation*}%
we have%
\begin{eqnarray*}
\int\limits_{I_{i}\backslash I_{i+1}}\left\vert \widetilde{D}_{n}^{\left(
t\right) }\right\vert d\mu &=&\int\limits_{I_{i}\backslash I_{i+1}}\max
\left\{ \alpha _{i}\left( n\right) ,2\alpha _{i}\left( m\right) \right\} d\mu
\\
&=&\frac{\max \left\{ \alpha _{i}\left( n\right) ,2\alpha _{i}\left(
m\right) \right\} }{2^{i+1}}.
\end{eqnarray*}%
Hence%
\begin{equation}
J_{1}=\sum\limits_{i=0}^{N-1}\frac{\max \left\{ \alpha _{i}\left( n\right)
,2\alpha _{i}\left( m\right) \right\} }{2^{i+1}}.  \label{J1}
\end{equation}%
For $J_{2}$ and $J_{3}$ we can write%
\begin{eqnarray}
J_{2} &=&\int\limits_{I_{N}\backslash I_{N+1}}\left\vert
D_{n}-2w_{m}D_{m}-2t_{N+1}\left( D_{n}-D_{2^{N}}\right) \right\vert d\mu
\label{J2} \\
&=&\frac{\left\vert 2^{N}-n^{\prime }-2m-2t_{N+1}\left( 2^{N}-n^{\prime
}-2^{N}\right) \right\vert }{2^{N+1}}  \notag \\
&=&\frac{\left\vert 2^{N}-n^{\prime }-2m+2t_{N+1}n^{\prime }\right\vert }{%
2^{N+1}},  \notag
\end{eqnarray}%
where $n=2^{N}+n^{\prime },n^{^{\prime }}<{2}^{N}$ and%
\begin{eqnarray}
J_{3} &=&\int\limits_{I_{N+1}}\left\vert D_{n}-2w_{m}D_{m}-2t_{N+1}\left(
D_{n}-D_{2^{N}}\right) \right\vert d\mu  \label{J3} \\
&=&\frac{\left\vert n-2m-2t_{N+1}\left( n-2^{N}\right) \right\vert }{2^{N+1}}%
.  \notag
\end{eqnarray}

Combining (\ref{J1-J3})-(\ref{J3}) we conclude that%
\begin{eqnarray*}
\left\Vert \widetilde{D}_{n}^{\left( t\right) }\right\Vert _{1}
&=&\sum\limits_{i=0}^{N-1}\frac{\max \left\{ \alpha _{i}\left( n\right)
,2\alpha _{i}\left( m\right) \right\} }{2^{i+1}} \\
&&+\frac{\left\vert 2^{N+1}-n-2m+2t_{N+1}\left( n-2^{N}\right) \right\vert }{%
2^{N+1}} \\
&&+\frac{\left\vert n-2m-2t_{N+1}\left( n-2^{N}\right) \right\vert }{2^{N+1}}%
.
\end{eqnarray*}

Set 
\begin{equation*}
A(n)=\{i:\left\vert n_{i}-n_{i-1}\right\vert =1)
\end{equation*}%
and%
\begin{equation*}
S\left( n\right) :=\sum\limits_{i\in A\left( n\right) }\frac{\alpha
_{i}\left( n\right) }{2^{i+1}}.
\end{equation*}

First, we prove that 
\begin{equation*}
S(n)\geq S(n(e)),
\end{equation*}

\bigskip where 
\begin{equation*}
n(e)=n_{N}2^{N}+...+n_{e+1}2^{e+1}+n_{e-1}2^{e}+....+n_{0}2^{1}.
\end{equation*}

and 
\begin{equation*}
n_{e}=n_{e-1}\neq n_{e+1}.
\end{equation*}

Set 
\begin{equation*}
A_{e}(n)=\{i:i\in A(n),i>e\}.
\end{equation*}

Then,%
\begin{eqnarray*}
S(n)-S(n(e)) &=&\sum_{i\in A_{e}(n)}\left( \frac{\alpha _{i}\left( n\right) 
}{2^{i+1}}\_\frac{\alpha _{i}\left( n(e)\right) }{2^{i+1}}\right) \\
&&+\sum_{i\in A(n)\backslash A_{e}(n)}\left( \frac{\alpha _{i}\left(
n\right) }{2^{i+1}}\_\frac{\alpha _{i+1}\left( n(e)\right) }{2^{i+2}}\right)
.
\end{eqnarray*}

Since%
\begin{equation*}
\frac{\alpha _{i}\left( n\right) }{2^{i+1}}\_\frac{\alpha _{i+1}\left(
n(e)\right) }{2^{i+2}}=0\text{ \ \ }\left( i\in A(n)\backslash
A_{e}(n)\right) ,
\end{equation*}

we get 
\begin{equation*}
S(n)-S(n(e))=\sum_{i\in A_{e}(n)}\left( \frac{\alpha _{i}\left( n\right) }{%
2^{i+1}}\_\frac{\alpha _{i}\left( n(e)\right) }{2^{i+1}}\right) .
\end{equation*}

We can write%
\begin{eqnarray*}
&&\alpha _{i}\left( n\right) -\alpha _{i}\left( n(e)\right) \\
&=&\left\vert n_{0}2^{0}+\cdots +n_{i-1}2^{i-1}-n_{i}2^{i}\right\vert \\
&&-\left\vert n\left( e\right) _{0}2^{0}+\cdots +n\left( e\right)
_{i-1}2^{i-1}-n\left( e\right) _{i}2^{i}\right\vert \\
&=&\left\vert n_{0}2^{0}+\cdots +n_{i-1}2^{i-1}-n_{i}2^{i}\right\vert \\
&&-\left\vert n_{0}2^{1}+\cdots +n_{e-1}2^{e}+n_{e+1}2^{e+1}+\cdots
+n_{i-1}2^{i-1}-n_{i}2^{i}\right\vert .
\end{eqnarray*}

Suppose that $n_{i}=1$. Then we can write%
\begin{eqnarray*}
&&\alpha _{i}\left( n\right) -\alpha _{i}\left( n(e)\right) \\
&=&2^{i}-n_{0}2^{0}-\cdots -n_{i-1}2^{i-1} \\
&&-(2^{i}-n_{0}2^{1}-\cdots -n_{e-1}2^{e}-n_{e+1}2^{e+1}-\cdots
-n_{i-1}2^{i-1}) \\
&=&\sum_{i=0}^{e-1}n_{i}2^{i}-n_{e}2^{e}.
\end{eqnarray*}

Let now consider the case when $n_{i}=0.$ Then we have%
\begin{equation*}
\alpha _{i}\left( n\right) -\alpha _{i}\left( n(e)\right) =-\left(
\sum_{j=0}^{e-1}n_{j}2^{j}-n_{e}2^{e}\right) .
\end{equation*}

So, we get following 
\begin{equation*}
\alpha _{i}\left( n\right) -\alpha _{i}\left( n(e)\right) =(2\left\vert
n_{i}-n_{e}\right\vert -1)\left\vert
\sum_{j=0}^{e-1}n_{j}2^{j}-2^{e}n_{e}\right\vert .
\end{equation*}

And finally we get 
\begin{eqnarray*}
&&S(n)-S(n(e)) \\
&=&\left\vert \sum_{j=0}^{e-1}n_{j}2^{j}-2^{e}n_{e}\right\vert \sum_{i\in
A_{e}(n)}\frac{(2\left\vert n_{i}-n_{e}\right\vert -1)}{2^{i+1}} \\
&>&\frac{1}{2^{e+2}}-\sum\limits_{i=e+2}^{\infty }\frac{1}{2^{i+1}}\geq 0.
\end{eqnarray*}

From the definition of function $V\left( n\right) $ it is easy to see that

\begin{equation*}
V(n)=V(n(e)).
\end{equation*}

If we continue this process it is easy to see that we will get

\begin{equation*}
n^{\prime }=\sum_{i=0}^{|n^{\prime }|-1}n_{i}^{\prime }2^{i},
\end{equation*}

where, for $0\leq i<|n^{\prime }|$\ \ \ $n_{i}^{\prime }+n_{i+1}^{\prime
}=1. $ Without lost of generality we can suppose that,$\ n_{0}^{\prime }=0$.
Hence%
\begin{equation*}
n^{\prime }=\left( 010101...\right) .
\end{equation*}%
Note, that 
\begin{equation*}
V(n)=V(n^{\prime }).
\end{equation*}%
and%
\begin{equation*}
S\left( n\right) \geq S\left( n^{\prime }\right) .
\end{equation*}%
Now, we calculate $S\left( n^{\prime }\right) $. We suppose that $%
V(n^{\prime })=2s$. It is easy to see that 
\begin{equation*}
\alpha _{2m}=2^{1}+2^{3}+\cdots +2^{2m-1}=\frac{2^{2m+1}-2}{3}
\end{equation*}%
and%
\begin{equation*}
\alpha _{2m-1}=\left\vert 2^{1}+2^{3}+\cdots +2^{2m-3}-2^{2m-1}\right\vert =%
\frac{2^{2m}+2}{3}.
\end{equation*}%
Then we can write%
\begin{eqnarray*}
S\left( n^{\prime }\right) &=&\sum\limits_{m=1}^{s}\frac{\alpha _{2m}}{%
2^{2m+1}}+\sum\limits_{m=1}^{s}\frac{\alpha _{2m-1}}{2^{2m}} \\
&=&\sum\limits_{m=1}^{s}\frac{1}{2^{2m+1}}\frac{2^{2m+1}-2}{3}%
+\sum\limits_{m=1}^{s}\frac{1}{2^{2m}}\frac{2^{2m}+2}{3} \\
&=&\frac{2}{3}s-\frac{2}{3}\sum\limits_{m=1}^{s}\frac{1}{2^{2m+1}}+\frac{2}{3%
}\sum\limits_{m=1}^{s}\frac{1}{2^{2m}} \\
&=&\frac{2}{3}s+\frac{1}{9}\left( 1-\frac{1}{2^{2s}}\right) .
\end{eqnarray*}%
Consequently,%
\begin{equation}
\frac{S\left( n\right) }{V\left( n\right) }\geq \frac{S\left( n^{\prime
}\right) }{V\left( n^{\prime }\right) }=\frac{1}{3}+\frac{1}{18s}\left( 1-%
\frac{1}{2^{2s}}\right) \geq \frac{1}{3}.  \label{low3}
\end{equation}

Since

\begin{equation*}
T(n,m)\cap A(m)=\varnothing
\end{equation*}

we can write 
\begin{eqnarray}
&&\sum\limits_{i=0}^{N-1}\frac{\max \left\{ \alpha _{i}\left( n\right)
,2\alpha _{i}\left( m\right) \right\} }{2^{i+1}}  \label{sandro} \\
&\geqslant &\sum\limits_{i\in A(m)\backslash \left\{ N\right\} }\frac{%
2\alpha _{i}\left( m\right) }{2^{i+1}}+\sum\limits_{i\in T(n,m)}\frac{\alpha
_{i}\left( n\right) }{2^{i+1}}  \notag \\
&\geq &\sum\limits_{i\in A(m)}\frac{2\alpha _{i}\left( m\right) }{2^{i+1}}%
+\sum\limits_{i\in T(n,m)}\frac{\alpha _{i}\left( n\right) }{2^{i+1}}-\frac{%
2\alpha _{N}\left( m\right) }{2^{N+1}}  \notag
\end{eqnarray}

From (\ref{low3}) we have 
\begin{equation}
\sum\limits_{i\in A(m)}\frac{2\alpha _{i}\left( m\right) }{2^{i+1}}=2S\left(
m\right) >\frac{2}{3}V(m).  \notag
\end{equation}

It is easy to see that if $n_{i}\neq n_{i-1}$ then 
\begin{equation*}
\alpha _{i}\left( n\right) \geq 2^{i-1}.
\end{equation*}

\bigskip So, we have 
\begin{equation}
\sum\limits_{i\in T(n,m)}\frac{\alpha _{i}\left( n\right) }{2^{i+1}}\geq 
\frac{1}{4}\left\vert T(n,m)\right\vert .  \label{15}
\end{equation}

Combining (\ref{l4}) and (\ref{l5}) we have%
\begin{equation}
\sum\limits_{i=0}^{N-1}\frac{\max \left\{ \alpha _{i}\left( n\right)
,2\alpha _{i}\left( m\right) \right\} }{2^{i+1}}\geq \frac{2}{3}V(m)+\frac{1%
}{4}\left\vert T(n,m)\right\vert -1.  \label{low1}
\end{equation}

On the other hand, we can write 
\begin{eqnarray}
&&\sum\limits_{i=0}^{N-1}\frac{\max \left\{ \alpha _{i}\left( n\right)
,2\alpha _{i}\left( m\right) \right\} }{2^{i+1}}  \label{low2} \\
&\geqslant &\sum\limits_{i\in A(n)\backslash \left\{ N,N+1\right\} }\frac{%
\alpha _{i}\left( n\right) }{2^{i+1}}+\sum\limits_{i\in T(m,n)}\frac{2\alpha
_{i}\left( m\right) }{2^{i+1}}  \notag \\
&&\frac{1}{3}V\left( n\right) +\frac{1}{2}\left\vert T\left( m,n\right)
\right\vert -\frac{\alpha _{N}\left( N\right) }{2^{N+1}}-\frac{\alpha
_{N+1}\left( N\right) }{2^{N+2}}  \notag \\
&&  \notag \\
&\geq &\frac{1}{3}V\left( n\right) +\frac{1}{2}\left\vert T\left( m,n\right)
\right\vert -1.  \notag
\end{eqnarray}

Combining (\ref{low1}) and (\ref{low2}) we have 
\begin{eqnarray}
&&\max (\frac{1}{2}\left\vert T(n,m)\right\vert +\frac{1}{3}V(n)-\frac{3}{2},%
\frac{1}{4}\left\vert T(n,m)\right\vert +\frac{2}{3}V(m))  \label{lowmain} \\
&\leqslant &\sum\limits_{i=0}^{N-1}\frac{\max \left\{ \alpha _{i}\left(
n\right) ,2\alpha _{i}\left( m\right) \right\} }{2^{i+1}}.  \notag
\end{eqnarray}

Now, we prove upper estimation. First, we prove 
\begin{eqnarray}
&&\sum\limits_{i=0}^{N-1}\frac{\max \left\{ \alpha _{i}\left( n\right)
,2\alpha _{i}\left( m\right) \right\} }{2^{i+1}}  \label{upper1} \\
&\leq &2\sum\limits_{i\in A(m)\cup T(n,m)}\frac{\max \left\{ \alpha
_{i}\left( n\right) ,2\alpha _{i}\left( m\right) \right\} }{2^{i+1}}.  \notag
\end{eqnarray}%
Suppose that $A(m)\cup T(n,m)=\left\{ r_{1<}r_{2}......<r_{s}\right\} .$
Then 
\begin{eqnarray*}
&&\sum\limits_{i=0}^{N-1}\frac{\max \left\{ \alpha _{i}\left( n\right)
,2\alpha _{i}\left( m\right) \right\} }{2^{i+1}} \\
&=&\sum_{i=1}^{s-1}\sum_{j=r_{i}}^{r_{i+1}-1}\frac{\max \left\{ \alpha
_{j}\left( n\right) ,2\alpha _{j}\left( m\right) \right\} }{2^{j+1}}.
\end{eqnarray*}%
Let $n_{i}=n_{i+1}$ for some $i$. Then it is easy to see that 
\begin{eqnarray*}
\alpha _{i+1}\left( n\right) &=&\left\vert
2^{i+1}n_{i+1}-2^{i}n_{i}..-2^{0}n_{0}\right\vert \\
&=&\left\vert 2^{i}n_{i}-2^{i-1}n_{i-1}..-2^{0}n_{0}\right\vert \\
&=&\alpha _{i}\left( n\right) .
\end{eqnarray*}

Consequently, 
\begin{eqnarray*}
&&\sum_{j=r_{i}}^{r_{i+1}-1}\frac{\max \left\{ \alpha _{j}\left( n\right)
,2\alpha _{j}\left( m\right) \right\} }{2^{j+1}} \\
&\leq &\frac{\max \left\{ \alpha _{r_{i}}\left( n\right) ,2\alpha
_{r_{i}}\left( m\right) \right\} }{2^{r_{i}}}.
\end{eqnarray*}

Hence (\ref{upper1}) is proved.

Since%
\begin{equation*}
\alpha _{i}\left( n\right) \leq 2^{i}\text{ \ }\left( i\in \mathbb{N}\right)
\end{equation*}%
and%
\begin{equation*}
\alpha _{i}\left( n\right) \leq 2^{i-1}\text{ \ }\left( i\notin A\left(
n\right) \right)
\end{equation*}%
from (\ref{upper1}) we have%
\begin{eqnarray}
&&\sum\limits_{i=0}^{N-1}\frac{\max \left\{ \alpha _{i}\left( n\right)
,2\alpha _{i}\left( m\right) \right\} }{2^{i+1}}  \label{u1} \\
&\leq &2\sum\limits_{i\in A(m)}\frac{\max \left\{ \alpha _{i}\left( n\right)
,2\alpha _{i}\left( m\right) \right\} }{2^{i+1}}  \notag \\
&&+2\sum\limits_{i\in T(n,m)}\frac{\max \left\{ \alpha _{i}\left( n\right)
,2\alpha _{i}\left( m\right) \right\} }{2^{i+1}}  \notag \\
&\leq &2\left\vert A\left( m\right) \right\vert +\left\vert T\left(
n,m\right) \right\vert .  \notag
\end{eqnarray}

It is easy to see that

\begin{eqnarray}
&&\frac{\left\vert 2^{N+1}-n-2m+2t_{N+1}\left( n-2^{N}\right) \right\vert }{%
2^{N+1}}  \label{u2} \\
&&+\frac{\left\vert n-2m-2t_{N+1}\left( n-2^{N}\right) \right\vert }{2^{N+1}}%
\leq 2.  \notag
\end{eqnarray}%
From (\ref{u1}) and (\ref{u2})%
\begin{equation}
L_{n}^{\left( t\right) }\leq 2V(m)+\left\vert T(n,m)\right\vert +2.
\label{u}
\end{equation}%
Combining (\ref{low1}), (\ref{low2}) and (\ref{u}) we complete the proof of
Theorem \ref{TS}.

\section{Uniformly boundedness of \ conjugate Fej\'er means}

The first result with respect to the a.e. convergence of the Walsh-Fejér
means $\sigma _{n}f$ is due to Fine \cite{Fi}. Later, Schipp \cite{Sc}
showed that the maximal operator $\sigma ^{\ast
}f:=\sup\limits_{n}\left\vert \sigma _{n}f\right\vert $ is of weak type (1,
1), from which the a. e. convergence follows by standard argument. Schipp
result implies by interpolation also the boundedness of $\sigma ^{\ast
}:L_{p}\rightarrow L_{p}\,\left( 1<p\leq \infty \right) $. This fails to
hold for $p=1$ but Fujii \cite{Fu} proved that $\sigma ^{\ast }$ is bounded
from the dyadic Hardy space $H_{1}$ to the space $L_{1}$ (see also Simon 
\cite{Si1}). Fujii's theorem was extended by Weisz \cite{WeAM1}. Namely, he
proved that the maximal operator $\sigma ^{\ast }f$ and the conjugate
maximal operator $\widetilde{\sigma }_{\ast }^{\left( t\right) }$ are
bounded from the martingale Hardy space $H_{p}$ to the space $L_{p}$ for $%
p>1/2$. Simon \cite{Si2} gave a counterexample, which shows that this
boundedness does not hold for $0<p<1/2.\,$ In \cite{GoBud} (see also \cite%
{GoSin}, \cite{Go}) the first author proved that the maximal operator $%
\widetilde{\sigma }_{\ast }^{\left( t\right) }$ is not bounded from the
Hardy space $H_{1/2}$ to the space $L_{1/2}.$

Weisz \cite{WeAM1}, \cite{WeAM2} considered the norm convergence of
conjugate Fejér means. In particular, the following is true

\begin{weisz}[Weisz]
If $t\in \mathbb{I}$ then 
\begin{equation*}
\left\Vert \widetilde{\sigma }_{n}^{\left( t\right) }f\right\Vert
_{H_{p}}\leq c_{p}\left\Vert f\right\Vert _{H_{p},\,\,\,}\left( f\in
H_{p}\right) ,
\end{equation*}%
whenever $p>1/2.$
\end{weisz}

Since 
\begin{equation*}
\left\Vert f\right\Vert _{H_{1}}\lesssim 1+E\left( \left\vert f\right\vert
\log ^{+}\left\vert f\right\vert \right)
\end{equation*}%
Theorem A imply that the following is true.

\begin{weisz2}
Let $f\in L\log L$ and $t\in \mathbb{I}.$ Then 
\begin{equation}
E\left\vert \widetilde{\sigma }_{n}^{\left( t\right) }f\right\vert \lesssim
1+E\left( \left\vert f\right\vert \log ^{+}\left\vert f\right\vert \right) .
\label{llogl}
\end{equation}
\end{weisz2}

On the othar hand, for $t=0$ we have following estimation.

\begin{weisz3}
Let $f\in L_{1}.$ Then 
\begin{equation}
E\left\vert \widetilde{\sigma }_{n}^{\left( t\right) }f\right\vert
=E\left\vert \sigma _{n}f\right\vert \lesssim E\left\vert f\right\vert .
\label{l}
\end{equation}
\end{weisz3}

In this paper we will find necessary and sufficient condition on $t$ for
which the estimation (\ref{l}) holds for conjugate Fej\'er means. We also
prove that for dyadic irrational $t$, $L\log L$ is maximal Orlicz space for
which the estimation (\ref{llogl}) is valid.

\begin{theorem}
\label{bound}Let $t\in \mathbb{Q}$ and $f\in L_{1}$. Then%
\begin{equation*}
E\left\vert \widetilde{\sigma }_{n}^{\left( t\right) }f\right\vert \lesssim
E\left\vert f\right\vert .
\end{equation*}
\end{theorem}

\begin{theorem}
\label{unbound}Let $t\notin \mathbb{Q}$ and $Q\left( L\right) $ be an Orlicz
space for which%
\begin{equation*}
Q\left( L\right) \nsubseteqq L\log L.
\end{equation*}%
Then%
\begin{equation*}
\sup\limits_{A}\left\Vert \widetilde{\sigma }_{2^{A}}^{\left( t\right)
}\right\Vert _{Q\left( L\right) \rightarrow L_{1}}=\infty .
\end{equation*}
\end{theorem}

\begin{proof}[Proof of Theorem \protect\ref{bound}]
Let $2^{A}\leq n<2^{A+1}$. Then we can write%
\begin{equation}
\widetilde{\sigma }_{n}^{\left( t\right) }f=\frac{1}{n}\sum\limits_{m=1}^{A}%
\sum\limits_{k=2^{m-1}}^{2^{m}-1}\widetilde{S}_{k}^{\left( t\right) }f+\frac{%
1}{n}\sum\limits_{k=2^{A}}^{n-1}\widetilde{S}_{k}^{\left( t\right) }f.
\label{fejer}
\end{equation}

Since for $2^{m-1}\leq k<2^{m}$ \ $\left( E_{-1}f=0\right) $ 
\begin{eqnarray}
\widetilde{S}_{k}^{\left( t\right) }f &=&\rho _{0}\left( \rho \left(
t\right) \right) \widehat{f}\left( 0\right)
w_{0}+\sum\limits_{l=1}^{m-1}r_{l}\left( \rho \left( t\right) \right) \left(
E_{l}f-E_{l-1}f\right)  \label{S} \\
&&+r_{m}\left( \rho \left( t\right) \right) \left( S_{k}f-E_{m-1}f\right) 
\notag \\
&=&\sum\limits_{l=0}^{m-1}r_{l}\left( \rho \left( t\right) \right) \left(
E_{l}f-E_{l-1}f\right) +r_{m}\left( \rho \left( t\right) \right) \left(
S_{k}f-E_{m-1}f\right)  \notag
\end{eqnarray}%
from (\ref{fejer}) we have%
\begin{eqnarray*}
\widetilde{\sigma }_{n}^{\left( t\right) }f &=&\frac{1}{n}%
\sum\limits_{m=1}^{A}2^{m-1}\sum\limits_{l=0}^{m-1}r_{l}\left( \rho \left(
t\right) \right) \left( E_{l}f-E_{l-1}f\right) \\
&&+\frac{1}{n}\sum\limits_{m=1}^{A}r_{m}\left( \rho \left( t\right) \right)
\sum\limits_{k=2^{m-1}}^{2^{m}-1}\left( S_{k}f-E_{m-1}f\right) \\
&&+\frac{n-2^{A}}{n}\sum\limits_{l=0}^{A}r_{l}\left( \rho \left( t\right)
\right) \left( E_{l}f-E_{l-1}f\right) \\
&&+\frac{r_{A+1}\left( \rho \left( t\right) t\right) }{n}\sum%
\limits_{k=2^{A}}^{n-1}\left( S_{k}f-E_{m-1}f\right) \\
&=&\frac{1}{n}\sum\limits_{m=1}^{A}2^{m-1}\sum\limits_{l=0}^{m-1}r_{l}\left(
\rho \left( t\right) \right) \left( E_{l}f-E_{l-1}f\right)
\end{eqnarray*}%
\begin{eqnarray}
&&+\frac{1}{n}\sum\limits_{m=1}^{A}r_{m}\left( \rho \left( t\right) \right)
\left( 2^{m}\sigma _{2^{m}}f-2^{m-1}\sigma _{2^{m-1}}f\right)  \label{J1-J6}
\\
&&-\frac{1}{n}\sum\limits_{m=1}^{A}r_{m}\left( \rho \left( t\right) \right)
2^{m-1}E_{m-1}f  \notag \\
&&+\frac{n-2^{A}}{n}\sum\limits_{l=0}^{A}r_{l}\left( \rho \left( t\right)
\right) \left( E_{l}f-E_{l-1}f\right)  \notag \\
&&+\frac{r_{A+1}\left( \rho \left( t\right) \right) }{n}\left( n\sigma
_{n}f-2^{A}\sigma _{2^{A}}f\right)  \notag \\
&&-\frac{r_{A+1}\left( \rho \left( t\right) \right) }{n}\left(
n-2^{A}\right) E_{m-1}f  \notag \\
&=&:\sum\limits_{j=1}^{6}J_{j}f.  \notag
\end{eqnarray}%
Since%
\begin{equation}
E\left\vert E_{m}f\right\vert \leq E\left\vert f\right\vert  \label{e1}
\end{equation}%
and%
\begin{equation}
E\left\vert \sigma _{n}f\right\vert \lesssim E\left\vert f\right\vert
\label{e2}
\end{equation}%
we can write%
\begin{equation}
E\left\vert J_{j}f\right\vert \lesssim E\left\vert f\right\vert ,j=2,3,5,6.
\label{j2-6}
\end{equation}

For $J_{1}f$ we can write%
\begin{equation*}
J_{1}f=\frac{1}{n}\sum\limits_{m=1}^{A}2^{m-1}\widetilde{E}_{m}^{\left(
t\right) }f,
\end{equation*}%
where%
\begin{equation}
\widetilde{E}_{m}^{\left( t\right) }f:=\sum\limits_{l=0}^{m-1}r_{l}\left(
\rho \left( t\right) \right) \left( E_{l}f-E_{l-1}f\right) .  \label{E}
\end{equation}%
Since $E_{l}f=f\ast D_{2^{l}}$ we have%
\begin{equation}
\widetilde{E}_{m}^{\left( t\right) }f=f\ast
\sum\limits_{l=0}^{m-1}r_{l}\left( \rho \left( t\right) \right) \left(
D_{2^{l}}-D_{2^{l-1}}\right) :=f\ast \widetilde{D}_{2m}^{\left( t\right) },
\label{c1}
\end{equation}%
where%
\begin{equation*}
\widetilde{D}_{2^{m}}^{\left( t\right) }:=\sum\limits_{l=0}^{m-1}r_{l}\left(
\rho \left( t\right) \right) \left( D_{2^{l}}-D_{2^{l-1}}\right) .
\end{equation*}%
It is easy to see that $r_{l}\left( t\right) =\left( -1\right)
^{t_{l}}=1-2t_{l}.$ Then for $\widetilde{D}_{2^{l}}^{\left( t\right) }$ we
can write%
\begin{eqnarray*}
\widetilde{D}_{2^{l}}^{\left( t\right) } &=&\sum\limits_{l=0}^{m-1}\left(
1-2t_{l}\right) \left( D_{2^{l}}-D_{2^{l-1}}\right) \\
&=&D_{2^{m-1}}-2\sum\limits_{l=0}^{m-1}t_{l}\left(
D_{2^{l}}-D_{2^{l-1}}\right) \\
&=&\left( 1-2t_{m-1}\right) D_{2^{m-1}}-2\sum\limits_{l=0}^{m-2}\left(
t_{l}-t_{l+1}\right) D_{2^{l}}.
\end{eqnarray*}%
Consequently,%
\begin{equation}
f\ast \widetilde{D}_{2^{m}}^{\left( t\right) }=\left( 1-2t_{m-1}\right)
E_{m-1}f-2\sum\limits_{l=0}^{m-2}\left( t_{l}-t_{l+1}\right) E_{l}f.
\label{c2}
\end{equation}%
$t\in \mathbb{Q}$ imply that 
\begin{equation*}
\sum\limits_{l=0}^{\infty }\left\vert t_{l}-t_{l+1}\right\vert <\infty .
\end{equation*}%
From (\ref{e1}) we get 
\begin{equation*}
E\left\vert \widetilde{E}_{m}^{\left( t\right) }f\right\vert \lesssim \left(
\left\vert 2t_{m-1}-1\right\vert +\sum\limits_{l=0}^{\infty }\left\vert
t_{l}-t_{l+1}\right\vert \right) E\left\vert f\right\vert \lesssim
E\left\vert f\right\vert .
\end{equation*}%
Consequently, 
\begin{equation}
E\left\vert J_{1}f\right\vert =\frac{1}{n}\sum\limits_{m=1}^{A}2^{m-1}E\left%
\vert \widetilde{E}_{m}^{\left( t\right) }f\right\vert \lesssim E\left\vert
f\right\vert .  \label{J1}
\end{equation}%
Combining (\ref{J1-J6})-(\ref{J1}) we complete the proof of Theorem \ref%
{bound}.
\end{proof}

\begin{proof}[Proof of Theorem \protect\ref{unbound}]
Since $t\notin \mathbb{Q}$ there exists a sequences $\left\{ p_{i}:i\in 
\mathbb{P}\right\} $ and $\left\{ q_{i}:i\in \mathbb{P}\right\} $ such that%
\begin{equation*}
0\leq q_{1}\leq p_{1}<q_{2}\leq p_{2}<\cdots <q_{A}\leq p_{A}<\cdots
\end{equation*}%
and%
\begin{equation*}
t_{j}=\left\{ 
\begin{array}{l}
1,q_{i}\leq j\leq p_{i} \\ 
0,p_{i}<j<q_{i+1}%
\end{array}%
\right. ,i=1,2,....
\end{equation*}%
Set%
\begin{equation*}
\Delta _{A}:=I_{p_{A}+1}\left(
t_{0},...,t_{q_{1}-1},0,t_{q_{1}+1},...,t_{p_{1}-1},0,t_{p_{1}+1},...,t_{q_{A}-1},0,t_{q_{A}+1},...,t_{p_{A}-1},0\right) .
\end{equation*}

Define the function%
\begin{equation*}
f_{A}\left( x\right) :=2^{2A}\mathbb{I}_{\Delta _{A}}\left( x\right) ,
\end{equation*}%
where $\mathbb{I}_{E}$ is characteristic function of the set $E$. It is easy
to see that%
\begin{equation}
\mu \left( \Delta _{A}\right) =\frac{2^{p_{A}+1-2A}}{2^{p_{A}+1}}=\frac{1}{%
2^{2A}}.  \label{measure}
\end{equation}

We can write (see (\ref{J1-J6}) and (\ref{E}))%
\begin{eqnarray}
&&\widetilde{\sigma }_{2^{2p_{A}+1}}^{\left( t\right) }f_{A}  \label{f1-f3}
\\
&=&\frac{1}{2^{2p_{A}+1}}\sum\limits_{m=1}^{2p_{A}+1}2^{m-1}\widetilde{E}%
_{m}^{\left( t\right) }f_{A}  \notag \\
&&+\frac{1}{2^{2p_{A}+1}}\sum\limits_{m=1}^{2p_{A}+1}r_{m}\left( \rho \left(
t\right) \right) \left( 2^{m}\sigma _{2^{m}}f_{A}-2^{m-1}\sigma
_{2^{m-1}}f_{A}\right)  \notag \\
&&-\frac{1}{2^{2p_{A}+1}}\sum\limits_{m=1}^{2p_{A}+1}r_{m}\left( \rho \left(
t\right) \right) 2^{m-1}E_{m-1}f_{A}  \notag \\
&=&:F_{1}f_{A}+F_{2}f_{A}+F_{3}f_{A}.  \notag
\end{eqnarray}

From (\ref{e1}), (\ref{e2}) and (\ref{measure}) we have%
\begin{equation}
E\left\vert F_{j}f_{A}\right\vert \lesssim E\left\vert f_{A}\right\vert
\lesssim 1,j=2,3.  \label{f2-3}
\end{equation}

Set 
\begin{equation*}
\widetilde{\Delta }_{i}:=I_{p_{i}+1}\left(
x_{0},...,x_{q_{1}-1},0,x_{q_{1}+1},...,x_{p_{1}-1},0,x_{p_{1}+1},...,x_{q_{i}-1},0,x_{q_{i}+1},...,x_{p_{i}-1},1\right) .
\end{equation*}%
Suppose that $x\in \widetilde{\Delta }_{i}$ for some $i=1,2,...,A$. Then 
\begin{equation*}
E_{a}f_{A}\left( x\right) =2^{a}\int\limits_{I_{a}\left( x\right)
}f_{A}\left( s\right) d\mu \left( s\right) =0,a>p_{i}.\text{ \ \ }
\end{equation*}%
Therefore, for $m>p_{A}$ we obtain (see \ref{c2})%
\begin{eqnarray}
&&\widetilde{E}_{m}^{\left( t\right) }f_{A}\left( x\right)  \label{s3} \\
&=&-2\sum\limits_{a=0}^{p_{i}}\left( t_{a}-t_{a+1}\right) E_{a}f_{A}\left(
x\right)  \notag \\
&=&\sum\limits_{k=1}^{i}\left[ 2E_{q_{k}-1}f_{A}\left( x\right)
-2E_{p_{k}}f_{A}\left( x\right) \right]  \notag \\
&=&\sum\limits_{k=1}^{i}\left[ 2^{q_{k}+2A}\mu \left( I_{q_{k}-1}\left(
x\right) \cap \Delta _{A}\right) -2^{p_{k}+1+2A}\mu \left( I_{p_{k}}\left(
x\right) \cap \Delta _{A}\right) \right] .  \notag
\end{eqnarray}

it is easy to calculate that%
\begin{eqnarray*}
\mu \left( I_{q_{k}-1}\left( x\right) \cap \Delta _{A}\right) &=&\frac{%
2^{p_{A}-\left( q_{k}-1\right) -2\left( A-k+1\right) }}{2^{p_{A}+1}} \\
&=&2^{-q_{k}-2\left( A-k\right) -2}
\end{eqnarray*}%
and%
\begin{eqnarray*}
\mu \left( I_{p_{k}}\left( x\right) \cap \Delta _{A}\right) &=&\frac{%
2^{p_{A}-p_{k}-\left[ 2\left( A-k\right) +1\right] }}{2^{p_{A}+1}} \\
&=&2^{-p_{k}-2\left( A-k\right) -2}.
\end{eqnarray*}%
Hence, from (\ref{s3}) for $m\geq p_{A}+1$ and $x\in \widetilde{\Delta }%
_{i},i=1,2,...,A$ we have%
\begin{eqnarray}
&&\widetilde{E}_{m}^{\left( t\right) }f_{A}  \label{s4} \\
&=&\sum\limits_{k=1}^{i}\left[ 2^{q_{k}+2A}\cdot 2^{-q_{k}-2\left(
A-k\right) -2}-2^{p_{k}+1+2A}\cdot 2^{-p_{k}-2\left( A-k\right) -2}\right] 
\notag \\
&=&\sum\limits_{k=1}^{i}\left[ 2^{2k-2}-2^{2k-1}\right] =-\frac{2^{2i}-4}{3}.
\notag
\end{eqnarray}

Consequently, for $x\in \widetilde{\Delta }_{i}$ we get 
\begin{eqnarray*}
&&\frac{1}{2^{2p_{A}+1}}\sum\limits_{m=p_{A}+1}^{2p_{A}+1}2^{m-1}\widetilde{E%
}_{m}^{\left( t\right) }f_{A} \\
&=&-\frac{1}{2^{2p_{A}+1}}\sum\limits_{m=p_{A}+2}^{2p_{A}+1}2^{m-1}\frac{%
2^{2i}-4}{3} \\
&=&-\frac{1}{2^{2p_{A}+1}}\left( 2^{2p_{A}+1}-2^{p_{A}+1}\right) \frac{%
2^{2i}-4}{3}.
\end{eqnarray*}

Since%
\begin{eqnarray*}
E\left\vert \widetilde{E}_{m}^{\left( t\right) }f_{A}\right\vert &\leq
&E\left\vert E_{m}f_{A}\right\vert +2\sum\limits_{a=0}^{m-1}E\left\vert
E_{a}f_{A}\right\vert \\
&\lesssim &mE\left\vert f_{A}\right\vert \lesssim m
\end{eqnarray*}

and%
\begin{equation*}
F_{1}f_{A}=\frac{1}{2^{2p_{A}+1}}\sum\limits_{m=1}^{p_{A}+1}2^{m-1}%
\widetilde{E}_{m}^{\left( t\right) }f_{A}+\frac{1}{2^{2p_{A}+1}}%
\sum\limits_{m=p_{A}+2}^{2p_{A}+1}2^{m-1}\widetilde{E}_{m}^{\left( t\right)
}f_{A}.
\end{equation*}

we have%
\begin{eqnarray}
E\left\vert F_{1}f_{A}\right\vert &\geq &E\left\vert \frac{1}{2^{2p_{A}+1}}%
\sum\limits_{m=p_{A}+2}^{2p_{A}+1}2^{m-1}\widetilde{E}_{m}^{\left( t\right)
}f_{A}\right\vert  \label{f1} \\
&&-\frac{1}{2^{2p_{A}+1}}\sum\limits_{m=1}^{p_{A}+1}2^{m-1}E\left\vert 
\widetilde{E}_{m}^{\left( t\right) }f_{A}\right\vert  \notag \\
&\geq &\sum\limits_{i=1}^{A}\int\limits_{\widetilde{\Delta }_{i}}\left\vert 
\frac{1}{2^{2p_{A}+1}}\sum\limits_{m=p_{A}+1}^{2p_{A}+1}2^{m-1}\widetilde{E}%
_{m}^{\left( t\right) }f_{A}\right\vert d\mu  \notag \\
&&-\frac{1}{2^{2p_{A}+1}}\sum\limits_{m=1}^{p_{A}+1}2^{m-1}m  \notag \\
&\gtrsim &\sum\limits_{i=1}^{A}2^{2i}\mu \left( \widetilde{\Delta }%
_{i}\right) -\frac{p_{A}2^{p_{A}}}{2^{2p_{A}+1}}\gtrsim A.  \notag
\end{eqnarray}

Combining (\ref{f1-f3}), (\ref{f2-3}) and (\ref{f1}) we have%
\begin{equation}
E\left\vert \widetilde{\sigma }_{2^{2p_{A}+1}}^{\left( t\right)
}f_{A}\right\vert \gtrsim A\text{.}  \label{main}
\end{equation}

Let%
\begin{equation}
Q\left( 2^{2A}\right) \geq 2^{2A}\text{, \ \ \ \ }A\geq A_{0}.  \label{Q}
\end{equation}

By virtue of estimate (see (\cite{KR}))%
\begin{equation*}
\left\Vert f_{A}\right\Vert _{Q\left( L\right) }\leq 1+E\left\vert Q\left(
f_{A}\right) \right\vert ,
\end{equation*}%
from (\ref{measure}) and (\ref{Q}) we can write%
\begin{eqnarray*}
E\left\vert \widetilde{\sigma }_{2^{2p_{A}+1}}^{\left( t\right)
}f_{A}\right\vert &\leq &\left\Vert \widetilde{\sigma }_{2^{2p_{A}+1}}^{%
\left( t\right) }\right\Vert _{Q\left( L\right) \rightarrow L_{1}}\left\Vert
f_{A}\right\Vert _{Q\left( L\right) } \\
&\leq &\left\Vert \widetilde{\sigma }_{2^{2p_{A}+1}}^{\left( t\right)
}\right\Vert _{Q\left( L\right) \rightarrow L_{1}}\left( 1+E\left\vert
Q\left( f_{A}\right) \right\vert \right) \\
&=&\left\Vert \widetilde{\sigma }_{2^{2p_{A}+1}}^{\left( t\right)
}\right\Vert _{Q\left( L\right) \rightarrow L_{1}}\left( 1+Q\left(
2^{2A}\right) \mu \left( \Delta _{A}\right) \right) \\
&=&\left\Vert \widetilde{\sigma }_{2^{2p_{A}+1}}^{\left( t\right)
}\right\Vert _{Q\left( L\right) \rightarrow L_{1}}\left( 1+\frac{Q\left(
2^{2A}\right) }{2^{2A}}\right) \\
&\lesssim &\left\Vert \widetilde{\sigma }_{2^{2p_{A}+1}}^{\left( t\right)
}\right\Vert _{Q\left( L\right) \rightarrow L_{1}}\frac{Q\left(
2^{2A}\right) }{2^{2A}},A\geq A_{0}.
\end{eqnarray*}%
Consequently, by (\ref{main}) we have%
\begin{equation}
\left\Vert \widetilde{\sigma }_{2^{2p_{A}+1}}^{\left( t\right) }\right\Vert
_{Q\left( L\right) \rightarrow L_{1}}\gtrsim \frac{A2^{2A}}{Q\left(
2^{2A}\right) },A\geq A_{0}.  \label{low}
\end{equation}

The fact that%
\begin{equation*}
Q\left( L\right) \nsubseteqq L\log L
\end{equation*}%
is equalent to the condition%
\begin{equation*}
\overline{\lim\limits_{u\rightarrow \infty }}\frac{u\log u}{Q\left( u\right) 
}=\infty .
\end{equation*}%
Then there exists $\left\{ u_{k}:k\in \mathbb{P}\right\} $ such that%
\begin{equation*}
\lim\limits_{k\rightarrow \infty }\frac{u_{k}\log u_{k}}{Q\left(
u_{k}\right) }=\infty ,u_{k+1}>u_{k},k=1,2,...
\end{equation*}%
and a monotonically increasing sequence of positive integers $\left\{
A_{k}:k\in \mathbb{P}\right\} $ such that%
\begin{equation*}
2^{2A_{k}}\leq u_{k}<2^{2\left( A_{k}+1\right) }.
\end{equation*}%
Then we have%
\begin{equation*}
\frac{2^{2A_{k}}A_{k}}{Q\left( 2^{2A_{k}}\right) }\gtrsim \frac{u_{k}\log
u_{k}}{Q\left( u_{k}\right) }\rightarrow \infty \text{ \ \ as \ \ }%
k\rightarrow \infty \text{.}
\end{equation*}%
Then, from (\ref{low}) we conclude that%
\begin{equation*}
\sup\limits_{k}\left\Vert \widetilde{\sigma }_{2^{2p_{A_{k}}+1}}^{\left(
t\right) }\right\Vert _{Q\left( L\right) \rightarrow L_{1}}=\infty \text{.}
\end{equation*}%
Theorem \ref{unbound} is proved.
\end{proof}

\end{document}